\documentclass[reqno,oneside,12pt]{amsart}

\usepackage[english]{babel}
\usepackage[utf8]{inputenc}
\usepackage{amsfonts,amsmath,amsthm}
\usepackage{stmaryrd}
\usepackage{thmtools,thm-restate}
\usepackage[T1]{fontenc}
\usepackage{enumitem}
\usepackage{varioref}
\usepackage[all]{xy}
\usepackage{units}
\usepackage{hyperref}
\usepackage{graphicx}
\usepackage{amssymb}
\usepackage{mathabx}
\usepackage{times,mathptm, mathtools}
\usepackage{etoolbox}
\usepackage{tikz-cd}
\usepackage{mathrsfs}
\usepackage{mathdots}
\usepackage[left=2.4cm, right=2.4cm]{geometry}
\usepackage{comment}

\docsvlist{A,B,C,D,E,F,G,H,I,J,K,L,M,N,O,P,Q,R,S,T,U,V,W,X,Y,Z}

\docsvlist{A,B,C,D,E,F,G,H,I,J,K,L,M,N,O,P,Q,R,S,T,U,V,W,X,Y,Z}

\docsvlist{A,B,C,D,E,F,G,H,I,J,K,L,M,N,O,P,Q,R,S,T,U,V,W,X,Y,Z}

\docsvlist{A,B,C,D,E,F,G,H,I,J,K,L,M,N,O,P,Q,R,S,T,U,V,W,X,Y,Z}

\docsvlist{A,B,C,D,E,F,G,H,I,J,K,L,M,N,O,P,Q,R,S,T,U,V,W,X,Y,Z}

\renewcommand{\hat}{\widehat}
\newcommand{\R}{\mathbf{R}}
\newcommand{\C}{\mathbf{C}}
\newcommand{\Q}{\mathbf{Q}}

\newcommand{\N}{\mathbf{N}}
\newcommand{\A}{\mathbf{A}}

\newcommand{\K}{\mathbf{K}}
\renewcommand{\P}{\mathbf{P}}

\renewcommand{\div}{\operatorname{div}}
\DeclareMathOperator{\Cartier}{Cartier}
\DeclareMathOperator{\Jac}{Jac}
\DeclareMathOperator{\Pic}{Pic}
\DeclareMathOperator{\Div}{Div}
\DeclareMathOperator{\Gal}{Gal}
\DeclareMathOperator{\ord}{ord}
\newcommand{\OO}{\mathcal O}

\DeclareMathOperator{\Supp}{Supp}

\DeclareMathOperator{\Per}{Per}

\DeclareMathOperator{\spec}{Spec}

\DeclareMathOperator{\Aut}{Aut}

\DeclareMathOperator{\an}{an}

\DeclareMathOperator{\supp}{Supp}
\DeclareMathOperator{\ddc}{dd^c}

\newtheorem{thm}{Theorem}[section]
\newtheorem{thm*}{Theorem}
\newtheorem{bigthm}{Theorem}
\newtheorem{prop}[thm]{Proposition}
\newtheorem{cor}[thm]{Corollary}
\newtheorem{lemme}[thm]{Lemma}

\theoremstyle{definition}

\newtheorem{rmq}[thm]{Remark}

\AtBeginEnvironment{dfn}{%
  \setlist[enumerate]{label={(\roman*)}}
}

\AtBeginEnvironment{thm}{%
  \setlist[enumerate,1]{label={(\arabic*)}}
}

\AtBeginEnvironment{prop}{%
  \setlist[enumerate,1]{label={(\arabic*)}}
}

\AtBeginEnvironment{lemme}{%
  \setlist[enumerate,1]{label={(\arabic*)}}
}

\begin{document}

\title[Uniform bound for families of regular plane polynomial automorphisms]{Uniform bound on common periodic points for families of regular plane polynomial automorphisms}
\author{Marc Abboud, Yugang Zhang}
\address{Marc Abboud, Institut de mathématiques, Université de Neuchâtel
\\ Rue Emile-Argand 11 CH-2000 Neuchâtel}
\email{marc.abboud@normalesup.org}
\address{Yugang Zhang, Institut de Mathématiques de Bourgogne (UMR IMB), SCIENCES MIRANDE
\\ 9 AVENUE ALAIN SAVARY}
\email{yugang.zhang@ube.fr}

\thanks{The first author acknowledges support by the Swiss National Science Foundation Grant “Birational transformations of higher dimensional varieties” 200020-214999. The second author acknowledges support by the programme ATRACT of Région Bourgogne-Franche-Comté under the project ADYAUS (RECH-ATRAC-000012) and the French National Research Agency under the project DynAtrois (ANR-24-CE40-1163)}

\maketitle

\begin{abstract}
    Given two one-dimensional families $f$ and $g$ of regular plane polynomial automorphisms parameterised by an algebraic curve $B$, all defined over some number field $K$, such that one of them is dissipative, we prove that at any parameter $b\in B(\C)$, either $f_b$ and $g_b$ share a common iterate, or the number of their common periodic points $\Per(f_b) \cap \Per(g_b)$ is bounded by a uniform constant $D$ (independent of the parameter $b$). We thus extend a result of Mavraki and Schmidt for rational maps to our setting.
\end{abstract}

\section{Introduction}\label{sec:intro}
\subsection{Main results}
A \emph{regular plane (polynomial) automorphism} is a polynomial automorphism $f\in \mathrm{Aut}(\A^2)$ such that its indeterminacy point $p_+$ in $\P^2$ is distinct from the indeterminacy point $p_-$ of its inverse $f^{-1}$ (\cite{zbMATH01908328}). Any plane polynomial automorphism with its \emph{first dynamical degree} (\cite{friedlandDynamicalPropertiesPlane1989,zbMATH05004661}) $\lambda(f)>1$ (equivalently, with positive topological entropy if over $\C$) is conjugate (in $\mathrm{Aut}(\A^2)$) to such a map (\cite{friedlandDynamicalPropertiesPlane1989}). Essentially they are the only polynomial automorphisms of the affine plane with rich dynamics up to conjugation.

An \emph{(arithmetic) family $f=(f_t)_{t\in B}$ of regular plane polynomial automorphisms parametrized by $B$} consists of a number field $K$, an algebraic curve $B$, and a morphism $f\colon B\times\A^2 \to B\times\A^2$, all defined over $K$, such that for every parameter $b\in B(\overline{K})$, the specialization $f_b$ is a regular plane automorphism. Fix an archimedean place so that there is an embedding $K\to \C$.  Associated to such a family is the \emph{Jacobian map} $\Jac(f)\colon B \to \A^1_\C$, defined by $\Jac(f)(b)=\Jac(f_b)$. The family is called \emph{dissipative} (resp. \emph{conservative}) if $|\Jac(f)(b)|<1$ (resp. $|\Jac(f)(b)|=1$) for all $b\in B(\overline{K})$. 

As a corollary of the seminal work~\cite{dujardinDynamicalManinMumford2017} of Dujardin and Favre (on Manin-Mumford problem), if $f$ and $g$ are defined over $\C$, and $|\Jac(f)|\neq 1$, then they have only finitely many common periodic points unless they share a common iterate. It is therefore natural, and also motivated by earlier works (c.f., Sect.~\ref{subsection: intro_background}) in arithmetic and complex dynamics, to ask whether there exists a uniform bound on the number of common periodic points. We partially address this question in the one-dimensional case.

\begin{bigthm}\label{bigthmA}
    Let $f$ and $g$ be two arithmetic families of regular plane polynomial automorphisms and suppose $f$ is dissipative. Then there exist a positive constant $D> 0$ depending only on $f$ and $g$, and two positive integers $N,M>0$, such that for any $b \in B ( \C)$, either
  \begin{equation*}
    \# \Per (f_b) \cap \Per (g_b) \leq D,
    \label{eq:uniformperiodicpoints}
  \end{equation*}
  or
  \begin{equation}\label{eq:equalitynm}
      f^N_b = g^M_b. 
  \end{equation}
  And the equality~\eqref{eq:equalitynm} can appear for at most finitely many parameters unless $f^N = g^M$ globally.
\end{bigthm}
We can not expect to have uniform bound over all parameters. Consider for example 
        \begin{align*}
            f_b(x,y)\coloneqq (y,y^2+b-1/2\,x) \ \ \ \mathrm{and}\ \ \ g_b(x,y)\coloneqq (y,y^2-(1/2 +b)x).
        \end{align*}
Then over $b=0$, $f_0=g_0$.

Two regular plane polynomial automorphisms $f$ and $g$ are said to be \emph{(dynamically) independent} if there exist no integers $n,m\in\N^*$ such that $f^n=g^m$. Theorem~\ref{bigthmA} is a consequence of the following, where we relax the condition on the Jacobians, but we lose some information on the exceptional parameters.

\begin{bigthm}\label{bigthmB}
    Let $f$ and $g$ be two dynamically independent families of regular plane polynomial automorphisms. Suppose that the set 
    $\{|\Jac(f)|=1 \} \cap \{|\Jac(g)|=1 \} \subset B(\C)$ is discrete. Then there exist a positive constant $D> 0$, depending only on $f$ and $g$, such that for all but finitely many $b \in B (\C)$, 
  \begin{equation}
    \# \Per (f_b) \cap \Per (g_b) \leq D,
    \label{eq:uniformperiodicpoints}
  \end{equation}
\end{bigthm}
\begin{rmq}
    \begin{enumerate}
        \item The set $\{|\Jac(f)|=1 \} \cap \{|\Jac(g)|=1 \}$ is discrete if and only if it is non-polar.
        \item We actually prove a stronger form for the small fibered canonical height points, see Theorem~\ref{thm:small-height}. In the case where we don't have any control over the Jacobians, we have the less strong result Theorem~\ref{mainthm2}.
    \end{enumerate}
\end{rmq}

\subsection{Background and the strategy of Main theorems}\label{subsection: intro_background}
The quadratic polynomials $f_c \coloneqq z^2 + c$, $c \in \mathbf C$, despite their simple form, exhibit highly intricate dynamical behavior. It is known (\cite{bakerPreperiodicPointsUnlikely2011,yuanArithmeticHodgeIndex2017}) that for any $c_1 \neq c_2$, the set of common preperiodic points of $f_{c_1}$ and $f_{c_2}$ is finite . One of the first uniformity results in this dynamical setting is the seminal work~\cite{zbMATH07594475} of DeMarco, Krieger, and Ye , who proved the existence of a uniform bound on the number of common preperiodic points for all $c_1 \neq c_2$. They further conjectured that for any two complex rational maps $f$ and $g$ on the Riemann sphere, either the sets of preperiodic points coincide or the number of common preperiodic points is bounded by a uniform constant. Remark that the case of the family of flexible latt\`{e}s maps \cite{zbMATH07190307} proved by the same authors (before \cite{zbMATH07594475}) implies the uniform Manin-Mumford bound for a two-dimensional complex family of genus 2 curves.

In this direction, Mavraki and Schmidt~\cite{zbMATH08050359} later established the conjecture for one-dimensional arithmetic families of rational maps, while DeMarco and Mavraki~\cite{zbMATH07793888} showed that the conjecture holds on a Zariski open set in the product space of all pairs of rational maps $f$ and $g$ on $\P^1$ over $\C$. Moreover, the strategy in~\cite{zbMATH07793888} permits proving another conjecture of Bogomolov, Fu, and Tschinkel~\cite{zbMATH07085159} about the uniform bound on common torsion points for pairs of elliptic curves, that was first showed by Poineau~\cite{zbMATH08060553}. Ang and Yap\cite{zbMATH08143040} looked at this problem from a statistical perspective.\\

Given two families $f$ and $g$, one may regard them as a single family by considering the pair $(f,g)$. Their common preperiodic points, denoted $\mathrm{Prep}_{f,g}$, may then be identified with the intersection of the set of preperiodic points of $(f,g)$ and the diagonal. Following the strategy developed by DeMarco, Mavraki, and Schmidt, our goal is to show that the set $\mathrm{Prep}_{f,g}$ is not Zariski dense (after passing to suitable fibered powers) using Yuan–Zhang’s equidistribution theorem~\cite{yuanAdelicLineBundles2026} (equivalently, we could also have used Gauthier's theorem~\cite{zbMATH08144876} in our case). The idea of using arithmetic equidistribution theory originates in Ullmo’s proof of the Bogomolov conjecture for curves~\cite{zbMATH01192412}, which was later extended to abelian varieties by S.~Zhang~\cite{zbMATH01192411}. See also, for example, \cite{zbMATH07582363,arXiv:2101.10272,ji2026geometricapproachuniformboundedness} and Gao’s survey~\cite{arXiv:2104.03431}, together with the references therein, for some other uniformity results in arithmetic geometry.

To this end, we need to verify a non-degeneracy condition for the diagonal. This is achieved in the key Proposition~\ref{prop:current-not-vanishing}, where we show that the diagonal is degenerate if and only if the families $f$ and $g$ are dynamically independent or \emph{simultaneously isotrivial}. Proposition~\ref{prop:current-not-vanishing} relies on three main ingredients: an intersection inequality (Proposition~\ref{prop:inequality-intersection-number}) derived from Yuan–Zhang’s theory of adelically metrized line bundles~\cite{yuanAdelicLineBundles2026}; a characterization of the support of the equilibrium measure for regular plane polynomial automorphisms at all places by Dujardin and Favre~\cite{dujardinDynamicalManinMumford2017}; and the geometric height theory for families of regular plane polynomial automorphisms developed by Gauthier and Vigny~\cite{zbMATH07714404}. Although in $\P^n$ for $n \geq 2$, it's not true that either $f$ and $g$ have the same set of preperiodic points or there exists a uniform bound (see a recent conjecture~\cite{demarco2026geometrypreperiodic}), a rigidity result~\cite{dujardinDynamicalManinMumford2017} of Dujardin and Favre concerning the forward and backward Green currents of regular plane polynomial automorphisms allows us to conclude in our case.

Finally, let us mention that in the case of rational maps, or more generally of polarized endomorphisms, Gauthier and Vigny~\cite{gauthierGeometricDynamicalNorthcott2025} developed a systematic treatment of the geometric canonical height function using pluripotential theory. This approach can be seen explicitly in \cite{zbMATH07793888,zbMATH08050359}, and has also strongly inspired the present work.

\section*{Acknowledgments} We would like to thank Thomas Gauthier for discussions related to this problem. The second author is very grateful to Thomas for his invitation at the Institut de mathématique d'Orsay.

\section{Adelic divisors and line bundles over quasiprojective varieties}\label{sec:}
\subsection{$b$-divisors}\label{subsec:}
Let $U$ be a quasiprojective variety over some field $K$. A \emph{completion} of $U$ is a projective variety $X$ over
$K$ containing $U$ as a dense open subset. If $X,Y$ are two completions of $U$ then we have a canonical birational map
$X \dashrightarrow Y$ induced by the open subset $U$. A \emph{Cartier} $b$-divisor over $U$ is an equivalence class
of a couple $(X,D)$ where $X$ is a projective completion of $U$ and $D$ is a Cartier divisor over $X$, where $(X,D)$ and
$(X', D')$ are equivalent if there exists $\pi : Y \rightarrow X, \omega : Y \rightarrow X'$ such that $\pi^* D = \omega^*
D'$. We write $\Cartier_b (U / K)$ for the space of Cartier $b$-divisors over $U$. Yuan and Zhang in
\cite{yuanAdelicLineBundles2026} have defined a topology on this space called the \emph{boundary topology}. It is
defined as follows, a \emph{boundary divisor} of $U$ is an effective divisor $D_0$ is some completion $X$ of $U$ such
that $\Supp D_0 = X \setminus U$. If $D \in \Cartier_b (U)$ we define the \emph{extended norm} of $D$ with respect to
$D_0$ as 
\begin{equation}
  \parallel D \parallel_{D_0} := \inf \left\{ \epsilon > 0: - \epsilon D_0 \leq D \leq \epsilon D_0 \right\}.
  \label{<+label+>}
\end{equation}
The boundary topology actually does not depend on the choice of the boundary divisor $D_0$. We can take the completion
of $\Cartier_b (U / K)$ with respect to this pseudo-norm. An element in the completion will be called a \emph{$b$-divisor}
over $U$ and we write $\Div_b (U / K)$ for the set of $b$-divisors over $U$. Concretly, an element $D \in \Div_b (U / K)$ is the
data of a sequence $D_i \in \Cartier_b (U)$ such that for every $i, j \geq 0, {D_i}_{|U} = {D_j}_{|U}$ and there exists
a sequence of positive rational numbers $\epsilon_i$ converging to zero such that 
\begin{equation}
  \forall j \geq i, \quad - \epsilon_i D_0 \leq D_j - D_i \leq \epsilon_i D_0.
  \label{<+label+>}
\end{equation}
Such a sequence $(D_i)$ is called a \emph{Cauchy sequence}.
We say that $D \in \Div_b (U / K)$ is strongly nef if it is the limit of a Cauchy sequence of nef Cartier divisors. If
$D_1, \dots, D_n$ are strongly nef divisors over $U$ which are limits of nef Cartier divisors $D_{i,k}$, then one can
check that the intersection numbers
\begin{equation}
  D_{1,k} \cdots D_{n,k}
  \label{<+label+>}
\end{equation}
converge towards a nonnegative real number that we denote by $D_1 \cdots D_n$. This is because if $E_1, \dots, E_n$ and
$E_1 ', \dots., E_n'$ are nef Cartier $b$-divisors such that $E_i \geq E_i'$ then 
\begin{equation}
  E_1 \cdots E_n \geq E_1 ' \cdots E_n '.
  \label{<+label+>}
\end{equation}
\subsection{Over a metrised field}\label{subsec:}
Let $K_v$ be a complete field with respect to a non-trivial absolute value. If $v$ is archimedean then we have that
$K_v = \R$ or $\C$. If $v$ is non-archimedean, we write $\OO_v$ for the valuation ring of $K_v$. If $X$ is a
quasiprojective variety over $K_v$ we write $X^{\an}$ for its Berkovich analytification with respect to $v$. If
$K_v = \C$, then this is exactly the set of closed points $X(\C)$. If $v$ is non-archimedean then this is a locally
compact Hausdorff topological space which contains the set $X(\overline K_v)$ modulo the Galois action as a dense
subset. If the reader is not familiar with Berkovich spaces it is enough for this paper to think of it as $X(\overline
K_v)$ modulo the Galois action. Let $X$ be a projective variety over $K_v$ and $L$ be a $\Q$-line bundle. A
\emph{continuous metric} of $L$ over $X^{\an}$ is the data of a metric for every $x \in X^{\an}$ of the stalk of $L$ at $x$ and
such that for every local section $s$ of $L$ defined over an open subset $V \subset X^{\an}$ we have that the function 
\begin{equation}
  x \in V \mapsto - \log \left| s \right|_x
  \label{<+label+>}
\end{equation} is continuous. A \emph{metrised line bundle} $\overline L$ is the data of $\Q$-line bundle and metric on
it. Let $D$ be a $\Q$-Cartier divisor over $X$, a \emph{Green function} of $D$ over $X$ is a
continuous function $g : X^{\an} \setminus \supp D \rightarrow \R$ such that for every $q \in \supp D$ if $\xi$ is a
local equation of $D$ at $q$ then the function 
\begin{equation}
  x \mapsto g + \log \left| \xi (x) \right|
  \label{<+label+>}
\end{equation}
extends to a continuous function at $q$. A \emph{metrised divisor} $\overline D = (D,g)$ is the data of a $\Q$-Cartier divisor
$D$ and a Green function $g$ of $D$. Every metrised divisor $\overline D$ defines a metrised line bundle $\OO_X (\overline
D)$ canonically by setting $- \log \left| s_D \right| = g$ where $s_D$ is the canonical rational section of $\OO_X
(D)$ such that $\div (s_D) = D$. Conversely, if $\overline L$ is a metrised line bundle, then any rational section
$s$ of $L$ defines a metrised divisor $\hat \div (s)$ by setting $D = \div(s)$ and $g = - \log \left| s \right|$. We say
that $\overline D = (D,g)$ is \emph{effective} if $g \geq 0$, this implies in particular that the divisor $D$ is
effective, in that case we write $\overline D \geq 0$.

Suppose now that $v$ is non-archimedean. Let $D$ be a $\Q$-Cartier divisor over a projective variety $X$ over $K_v$. A
\emph{model} of $X$ is a projective variety $\sX$ over $\OO_v$ such that $X = \sX \times_{\spec \OO_v} \spec K_v$. A
model of $(X,D)$ is the data of a model $\sX$ of $X$ over $\OO_v $ and $\Q$-Cartier divisor $\sD$ over $\sX$ such that
$\sD_{|X} = D$. Any such model divisor induces a metrised divisor $\overline \sD$ as follows: there is a canonical
reduction map $r : X^{\an} \rightarrow \sX$ and if $x \in X^{\an} \setminus \Supp D$ we define $G_{\overline \sD}
(x) = - \log \left| \xi (x) \right|$ where $\xi$ is a local equation of $\sD$ at $r(x)$. There is a similar definition of
model metrised line bundles. If $X$ is of dimension $n$ and $\overline \sL_1, \dots, \overline \sL_n$ are model metrised
line bundles then they define a signed measure 
\begin{equation}
  c_1 (\overline \sL_1)\cdots c_1 (\overline \sL_n)
  \label{<+label+>}
\end{equation}
first defined by Chambert-Loir in \cite{chambert-loirMesuresEquidistributionEspaces2003}. If every $\sL_i$ is nef, then
this is a positive measure of total mass $L_1 \cdots L_n$.

If $K_v = \C$ and $\overline L$ is a metrised line bundle with a plurisubharmonic (psh) metric then there is a
$(1,1)$-closed positive current $c_1 (\overline L)$ associated to $L$ defined locally by 
\begin{equation}
  c_1 (\overline L) = dd^c \left( - \log \left| s \right| \right)
  \label{<+label+>}
\end{equation}
where $s$ is a local non vanishing section of $L$ and by Bedford-Taylor theory we have a well defined positive measure 
\begin{equation}
  c_1 (\overline L_1) \cdots c_1 (\overline L_n)
  \label{<+label+>}
\end{equation}
which is also of total mass $L_1 \cdots L_n$. Over $K_v = \C$ we make the convention that any metrised line bundle is
model. It is \emph{nef} is the associated metric is psh.

If $U$ is a quasiprojective variety over $K_v$ of dimension $n$, a \emph{model metrised divisor} over $U$ is a model metrised divisor over
any completion $X$ of $U$. Following Yuan and Zhang in \cite{yuanAdelicLineBundles2026} we can also define the boundary
topology on the space of model metrised divisors of $U$. In this setting a metrised boundary divisor will be a model metrised
divisor $\overline \sD_0$ such that $D_0$ is a boundary divisor and $G_{\overline \sD_0} > 0$. The definition of the
extended norm with respect to $\overline \sD_0$ is similar to the one in the previous paragraph we have 
\begin{equation}
  \parallel \overline \sD \parallel_{\overline \sD_0} := \inf \left\{ \epsilon > 0 : - \epsilon \overline \sD_0 \leq
  \overline \sD \leq \epsilon \overline \sD_0 \right\}. 
  \label{<+label+>}
\end{equation}
We write $\hat \Div (U/ K_v)$ for the completion of the set of model metrised divisors with respect to the boundary
topology. An element in that completion is called a \emph{metrised divisor} over $U$. Notice that if $\overline D = \lim
\overline \sD_i$ then at the level of Green functions we have that there exists $\epsilon_i > 0$ converging to zero such
that 
\begin{equation}
  - \epsilon_i G_{\overline \sD_0} \leq G_{\overline \sD_j} - G_{\overline \sD_i} \leq \epsilon_i G_{\overline \sD_0}.
  \label{<+label+>}
\end{equation}
In particular the sequence of Green functions $(G_{\overline \sD_i})$ converges uniformly locally towards a function
that we denote by $G_{\overline D}$.
We say that $\overline D \in \hat
\Div (U / K_v)$ is \emph{strongly nef} if it is the limit of a Cauchy sequence of nef model metrised divisors over $U$.
We see in particular that if we forget about the Green functions we recover the notion of $b$-divisors from the previous
paragraph, i.e we have a surjective group homomorphism $\hat \Div(U/ K_v) \rightarrow \Div_b (U / K_v)$ which sends
strongly nef metrised divisor to strongly nef $b$-divisors.

We can define similarly metrised line bundles over $U$ and strongly nef metrised line bundles. Furthermore, if
$\overline L_1, \dots, \overline L_n$ are strongly nef metrised line bundles which are limits of nef model metrised line
bundles $\overline \sL_{i, k}$ then the sequence of measures $c_1 (\overline \sL_{1,k}) \cdots c_1 (\overline
\sL_{n,k})$ weakly converges towards a measure that we denote by 
\begin{equation}
  c_1 (\overline L_1) \cdots c_1 (\overline L_n).
  \label{<+label+>}
\end{equation}
The total mass is given by the geometric intersection number.
\begin{thm}[\cite{yuanAdelicLineBundles2026, guoIntegrationFormulaChern2025}]\label{thm:total-mass-measure}
  Let $\overline L_1, \dots, \overline L_d$ be strongly nef metrised line bundles over a quasiprojective variety $U$ over
  a complete field $K_v$, then  
  \begin{equation}
    \int_{U^{\an,v}} c_1 (\overline L_1)_v\cdots c_1 (\overline L_d)_v = L_1 \cdots L_d.
    \label{<+label+>}
  \end{equation}
\end{thm}
\subsection{Over a projective variety}\label{subsec:}
Let $K$ be a number field and let $X$ be a projective variety over $K$. We write $\cM(K)$ for the set of normalised
absolute values over $K$. A \emph{model adelic line bundle} over $X$ is
the data of a model $\sX$ of $X$ over $\spec \OO_K$ and a $\Q$-line bundle $\sL$ over $\sX$ together with a metric of
$\sL_\C$ over $\sX (\C) = X(\C)$. Write $L = \sL_{|X}$, we write $\overline \sL$ for the data of $\sL$ and the metric of
$\sL_\C$.  For every
$v \in \cM(K)$, a model adelic line bundle defines a model metric of $L$ over $X^{\an, v}$ and a
\emph{height function} $h_{\overline \sL}: X(\overline K) \rightarrow \R$ defined by 
\begin{equation}
  \forall p \in X(\overline K), \quad h_{\overline \sL} (p) = \frac{1}{\# \Gal (\overline K / K) \cdot p} \sum_{q \in
  \Gal (\overline K / K) \cdot p} \sum_{v \in \cM(K)} n_v (- \log \left| \left| s (q) \right| \right|_{\overline \sL,
  v})
  \label{eq:height-function}
\end{equation}
where $s$ is a local section of $L$ not vanishing at $p$ and $n_v$ are rational numbers depending only on $v$ and $K$.
We say that
$\overline \sL$ is nef if $\sL$ is nef, the metric of $\sL_\C$ is plurisubharmonic and $h_{\overline \sL} \geq 0$. If $n
= \dim X$, and $\overline \sL_0, \dots,
\overline \sL_n$ are nef model adelic line bundles their \emph{arithmetic intersection number} is defined inductively as

\begin{equation}
  \overline \sL_0, \dots, \overline \sL_d = \left( \overline \sL_0 \cdots \overline \sL_{d-1}  \right)_{| \div
  (s_d)} + \sum_{v \in \cM(K)} \int_{X^{\an,v}} - \log \left| \left| s_d \right| \right|_{\overline \sL_d,v} c_1
  (\overline \sL_0)_v \cdots c_1 (\overline \sL_{d-1})_v
  \label{eq:intersection-number}
\end{equation}
where $s_d$ is a rational section of $L_d$ over $X$.

\begin{prop}\label{prop:}
  If $\overline \sL_0, \dots, \overline \sL_d$ are nef model adelic line bundles, then 
  \begin{equation}
    \overline \sL_0 \cdots \overline \sL_d \geq 0.
    \label{<+label+>}
  \end{equation}
\end{prop}
A \emph{model} adelic divisor is the data of a model $\sX$ and a $\Q$-Cartier divisor $\sD$ over $\sX$ and a Green
function of $\sD_\C$ over $\sX(\C)$. For a model adelic divisor $\overline \sD$, there is a unique associated adelic line bundle $\OO_X (\overline
\sD)$ up to isometry. The height function then can be easily computed: 
\begin{equation}
  \forall p \in (X \setminus \Supp D) (\overline K), \quad h_{\overline \sD} (p) = \frac{1}{\# \Gal (\overline K / K)
  \cdot p} \sum_{q \in \Gal (\overline K / K) \cdot p} \sum_{v \in \cM(K)} n_v G_{\overline \sD, v} (q).  \label{<+label+>}
\end{equation}
We say that $\overline \sD$ is \emph{effective} if for every $v \in \cM(K), G_{\overline \sD, v} \geq 0$. This implies
in particular that the $\Q$-Cartier divisor $\sD$ is effective. We say that $\overline \sD$ is nef if $\OO_X (\overline
\sD)$ is.

\subsection{Over a quasiprojective variety}\label{subsec:}
Let $U$ be a quasiprojective variety over a number field $K$. An adelic line bundle over $U$ is a Cauchy sequence of model
adelic line bundles over projective compactifications of $U$ for the boundary topology defined by Yuan and Zhang in
\cite{yuanAdelicLineBundles2026}. An adelic divisor over $U$ is a Cauchy sequence of
model adelic divisors $\overline \sD_i$ over projective compactifications of $U$ such that $\forall i, j,
{D_i}_{|U} = {D_j}_{|U}$ with respect to the boundary topology. We write $\hat \Div (U/ \OO_K)$ for the set of adelic
divisors over $U$ and $\hat \Pic(U/ \OO_K)$ for the set of adelic line bundles over $U$. If $\overline L$
is an adelic line bundle over $U$ which is the limit of a sequence of model adelic line bundles $\overline \sL_i$, then
it defines a height function $h_{\overline L} : U (\overline K) \rightarrow \R$ by
\begin{equation}
  h_{\overline L} := \lim_i h_{\overline \sL_i}.
  \label{<+label+>}
\end{equation}
An adelic line bundle $\overline L$ is \emph{strongly nef} if it is the limit of nef model adelic line bundles. If
$\overline L_0, \dots, \overline L_d$ are strongly nef, then the \emph{arithmetic intersection number} is defined as
\begin{equation}
  \overline L_0 \cdots \overline L_d := \lim_i \overline \sL_{0,i}, \cdots, \overline \sL_{d,i} \geq 0.
  \label{<+label+>}
\end{equation}
Let $v \in \cM(K)$ be a place of $K$ and write $K_v$ for the completion of $K$ with respect to $v$. Then by restricting
everything to the place $v$ we get a morphism $\hat \Div (U/ \OO_K) \rightarrow \hat \Div (U_{K_v} / K_v)$ which sends
strongly nef adelic divisors to strongly nef metrised divisors where $U_{K_v} = U \times_{\spec K} \spec K_v$.
We also have the measure $c_1 (\overline L_1)_v \cdots c_1 (\overline L_d)_v$ defined over $U^{\an,v}$ as the weak limit
of the measures $c_1 (\overline \sL_{1} )_v \cdots c_1 (\overline \sL_d)_v$ for every $v \in \cM(K)$.

If $\overline D$ is an adelic divisor over $U$ which is the limit of $\overline \sD_i$ then for every $v \in
\cM(K)$, the Green functions $G_{\overline \sD_i , v}$ converge locally uniformly towards $G_{\overline D, v}$ the Green
function of $\overline D$. In particular, the height formula still holds 
\begin{equation}
  \forall p \in U(\overline K) \setminus \Supp D_{|U}, h_{\overline D} (p) = \frac{1}{\# \Gal (\overline K / K) \cdot p}
  \sum_{q \in \Gal (\overline K / K) \cdot p} \sum_v n_v G_{\overline D, v} (q).
    \label{eq:height}
\end{equation}
We say that $\overline D$ is \emph{effective} if it is the limit of a Cauchy sequence of effective model adelic divisors.

Finally, if $\overline L_1, \dots, \overline L_n$ are strongly nef adelic line bundles, then if we forget about the
metrics we get a sequence of line bundles $L_{k,i}$ and the geometric intersection numbers
\begin{equation}
  L_{1,i} \cdots L_{d,i}
  \label{<+label+>}
\end{equation}
also converge to a nonnegative number that we denote by $L_1 \cdots L_d$. In particular, we say that $L$ is \emph{big} if
$L^d > 0$, this geometric intersection number is called the \emph{degree} of $X$ with respect to $L$ and it is denoted
by $\deg_{\overline L} (X)$.

Finally we end this section by stating the arithmetic equidistribution theorem of Yuan and Zhang in a version suitable
for our needs. We say that a sequence of closed points if \emph{generic} if no subsequence is contained in a strict
closed subvariety.

\begin{thm}[\cite{yuanAdelicLineBundles2026}]\label{thm:arithmetic-equidistribution}
  Let $\overline L \in \hat \Pic (U/ \OO_K)$ be a strongly nef adelic line bundle such that $L^d > 0$. If $(p_n) \subset
  U(\overline K)$ is a generic sequence such that $h_{\overline L} (p_n) \rightarrow 0$, then for any $v \in
  \cM(K)$ the sequence of Dirac measures 
  \begin{equation}
    \delta_{n,v} = \frac{1}{\# \Gal (\overline K / K) \cdot p_n} \sum_{q \in \Gal (\overline K / K) \cdot p_n} \delta_{q,v}
    \label{<+label+>}
  \end{equation}
  weakly converges towards $c_1 (\overline L)_v$ over $U^{\an,v}$
\end{thm}
\begin{proof}
  This is the result of \cite{yuanAdelicLineBundles2026} except that we need to show that $\overline L^{d+1} = 0$ to
  apply the result of Yuan and Zhang. To do so we use the \emph{fundamental inequality} of Yuan and Zhang from
  \cite{yuanAdelicLineBundles2026} 
  \begin{equation}
    \sup_{V \subset U} \inf_{q \in V (\overline K)} h_{\overline L} (q) \geq \overline L^{d+1}
    \label{<+label+>}
  \end{equation}
  where the supremum is over open subset of $U$.
  Since the sequence $p_n$ is generic we get that $\overline L^{d+1} \leq 0$ and it is $\geq 0$ because $\overline L$ is
  strongly nef.
\end{proof}

The genericity condition is not too restrictive. From a Zariski dense sequence $(p_n)$ of closed points it is always
possible to extract a generic one. Indeed, there are countably many strict closed subvarieties defined over $\overline
K$ of $U$ since $K$ is countable, we enumerate them: $Z_1, Z_2, \cdots$, define $Z_i ' = Z_1 \cup Z_2 \cdots \cup Z_i$.
We define an extracting function $\sigma : \N \rightarrow \N$ as follows, pick $\sigma (1)$ such that $p_{\sigma
(1)} \not \in Z_1 = Z_1 '$. Suppose that $\sigma (1), \cdots, \sigma (n)$ has been constructed, we choose $\sigma
(n+1 ) > \sigma (n)$ such that $p_{\sigma (n+1)} \not \in Z_{n+1} '$ the subsequence $p_{\sigma (n)}$ is generic.

\subsection{In the geometric setting}\label{subsec:}
Suppose now that $K = k(B)$ is the function field of a projective curve $B$ over some field $k$. Yuan and Zhang have
also developped their theory of adelic line bundles and divisors in this setting. To distinguish with the arithmetic
case, we will call those \emph{geometric adelic line bundle and divisors}. We will write geometric adelic line bundle
and divisors with a tilde to distinguish them from arithmetic adelic line bundles and divisors which are written with a
bar. We write $\widetilde \Div(U/ k(B))$ and $\widetilde \Pic(U/k(B))$ for the set of geometric adelic divisors and line
bundles respectively. We briefly recall the theory.  Let $X$ be a projective variety over $K$. We write $\cM (K)$ for
the set of closed points of $B$. Each point $p \in \cM (K)$ defines an absolute value over $K$ by 
\begin{equation}
  \forall \lambda \in K, \quad \left| \lambda \right|_p = e^{- \ord_p (\lambda)}.
  \label{<+label+>}
\end{equation}
If $L$ is a $\Q$-line bundle over $X$, a \emph{model} of $L$ is the data of a model $\sX$ of $X$ over $B$ and a $\Q$-line bundle
$\widetilde \sL$ over $\sX$ such that $\widetilde \sL_{|X} = L$. It also defines a height function $h_{\widetilde \sL} : X (\overline K)
\rightarrow \R$ via 
\begin{equation}
  h_{\widetilde \sL} (p) = \Delta_p \cdot \widetilde \sL
  \label{<+label+>}
\end{equation}
where $\Delta_p$ is the closure of $p$ in $\sX$. However we can also interpret this using Berkovich spaces, indeed for
every $v \in \cM (K)$, $\widetilde \sL$ defines a metric of the line bundle $L$ over $X^{\an,v}$ and Formula
\eqref{eq:height-function} also holds. The geometric intersection number $\widetilde \sL_0 \cdots \widetilde \sL_n$ has also an expression of
the form of \eqref{eq:intersection-number}. We also have an analog for geometric model adelic divisors. In particular, if $\widetilde \sD$ is
a $\Q$-Cartier divisor over some model $\sX$ of $X$ restricting to a divisor $D$ over $X$, then for every $b \in
\cM(K)$, it induces a Green function of $D$ over $X^{\an, b}$.
Suppose now that $U$ is a quasiprojective variety over $K$, then a geometric adelic line bundle (resp. geometric adelic divisor) over $U$
is a Cauchy sequence of geometric model adelic line bundles (resp. geometric model adelic divisors) on projective models of $U$ over $B$
with respect to the boundary topology. In particular, if $\widetilde D$ is an adelic divisor over $U$ which is the limit
of $\widetilde \sD_i$, then the Green functions $G_{\widetilde D_i , b}$ converge uniformly locally over $U^{\an, b}$
towards the Green function $G_{\widetilde D, b}$ of $\widetilde D$ and the formulas \eqref{eq:height-function}, 
\eqref{eq:intersection-number} and \eqref{eq:height} also hold in the geometric setting.

\subsection{Restriction of adelic divisors and line bundles}\label{subsec:restriction-adelic-divisors}
Suppose $Y,X$ are quasiprojective varieties over a number field $K$ and let $f:Y \rightarrow X$ be a morphism, there is
a natural \emph{pullback operator} 
\begin{equation}
  f^* \hat \Pic (X/ \OO_K) \rightarrow \hat \Pic (Y / \OO_K), \quad f^* \hat \Div (X / \OO_K)\rightarrow \hat \Div
  (Y/ \OO_K).
  \label{<+label+>}
\end{equation}
We describe in one specific case that is useful for this paper. Let $X = \A^2_K$ and $C \subset X$ be an irreducible
curve. Suppose that $\overline D$ is a strongly nef adelic Green divisor over $\A^2_K$ such that $D_{|\A^2} = 0$. Write $\overline D
= \lim_i \overline \sD_i$ and let $\sX_i$ be a projective model over which $\sD_i$ is defined. Write $\sY_i$ for the
Zariski closure of $C$ in $\sX_i$, then $\sD_i$ restricts to a Cartier divisor over $\sY_i$ and the Green function
$G_{\overline D_i}$ restricts to a Green function ${G_{\overline D_i}}_{|C(\C)}$ of ${D_i}_{|\overline C}$ where
$\overline C$ is the Zariski closure of $C$ in $X_i$. In particular, the degree of $D_{|C}$ is equal to the limit 
\begin{equation}
    \deg(D_{|C}) = \lim_i \deg( {D_i}_{|\overline C}) = \lim_i D_i \cdot \overline C.
\end{equation}

We end this section with a computation of geometric intersection numbers  over a fibration of projective varieties. Using \cite[Example 20.3.3]{fultonIntersectionTheory1998}, we can state the following: If $q : X \rightarrow S$ is a flat projective morphism where $S$ is the spectrum of a discrete valuation ring and $D_1 \cdots D_n$ are Cartier divisors where $n = \dim_S X$, then 
\begin{equation}
    D_{1,\eta} \cdots D_{n,\eta} = D_{1,s} \cdots D_{n,s}
\end{equation}
where $\eta$ is the generic point of $S$ and $s$ is the closed point. From this, one can deduce the following proposition:

\begin{prop}\label{prop:intersection-number-fibration}
    Let $q :X \rightarrow Y$ be a morphism between irreducible smooth projective varieties over a field $K$ and $D_1, ..., D_n$ be $n$ Cartier divisors over $X$ and $E_1, ..., E_m$ be $m$ Cartier divisors over $Y$ with $n = \dim X - \dim Y, m = \dim Y$, then
    \begin{equation}
        D_1 \cdots D_n \cdot q^* E_1 \cdots q^* E_m = (D_{1,\eta} \cdots D_{n, \eta}) (E_1 \cdots E_m)
    \end{equation}
    where $\eta$ is the generic point of $Y$ and $D_{i,\eta}$ is the restriction of $D_i$ to the generic fiber.
\end{prop}

\begin{cor}\label{cor:intersection-number-fibration-quasiprojective}
    Let $q: X \rightarrow Y$ be a morphism between normal quasiprojective varieties over some field $K$, let $D_1, ...,
    D_n$ be strongly nef $b$-divisors over $X$ and let $E_1, ..., E_m$ be strongly nef $b$-divisors over $Y$ with $n =
    \dim X - \dim Y$ and $m = \dim Y$, then 
    \begin{equation}
        D_1 \cdots D_n \cdot q^*E_1 \cdots q^* E_m = (D_{1,\eta} \cdots D_{n,\eta} ) \cdot (E_1 \cdots E_m).
    \end{equation}
\end{cor}
\begin{proof}
    Let $D_{i,k}, E_{i,k}$ be nef model divisors converging to $D_i$ and $E_i$ respectively. We can assume that every
    $D_{i,k}$ is defined over the completion $X_k$ of $X$ and every $E_{i,k}$ defined over the same completion $Y_k$ and such that the morphism $q$ extends to a morphism $q:X_k \rightarrow Y_k$. Then, we have that 
    \begin{equation}
        D_1 \cdots D_n \cdot q^* E_1 \cdots q^* E_n = \lim_k D_{1,k} \cdots D_{n,k} \cdot q^* E_{1,k} \cdots q^* E_{n,k}.
    \end{equation}
    Now the term in the limit is equal by Proposition \ref{prop:intersection-number-fibration} to 
    \begin{equation}
        (D_{1,k,\eta} \cdots D_{n, k, \eta}) (E_{1,k} \cdots E_{m.k}) \xrightarrow[k \rightarrow + \infty]{} (D_{1,\eta} \cdots D_{n,\eta}) \cdot {E_1 \cdots E_m}.
    \end{equation}
\end{proof}

Another corollary of Fulton's result is the following:

\begin{prop} \label{prop:independent-intersection-number-fiber}
Let $q : X \rightarrow Y$ be a projective flat morphism of varieties, if $D_1, ..., D_e$ are Cartier divisors over $X$ where $e = \dim X  - \dim Y$, then for every $y \in Y$, the intersection number $D_{1,y} \cdots D_{e,y}$ of the Cartier divisors restricted to the fibery of $y$ is independent of $y$.
\end{prop}

\section{Invariant adelic divisor}\label{sec:invariant-adelic-divisor}
Let $f : \A^2 \rightarrow \A^2$ be a regular automorphism defined over any field $L$ of degree $\lambda(f)$. It follows from
\cite{boucksomDegreeGrowthMeromorphic2008, cantatDynamiqueAutomorphismesSurfaces2001,
favreDynamicalCompactificationsMathbf2011,abboudDynamicsEndomorphismsAffine2023} that there exist two
invariant $b$-divisors $\theta^\pm_{f,L} \in \Div_b(\A^2_L / L)$ unique up to multiplication by a positive constant such that 
\begin{equation}
  f^* \theta^+_{f,L} = \lambda (f) \theta^+_{f,L}, \quad (f^{-1})^* \theta^{-1}_{f,L} = \lambda(f) \theta^-_{f,L}.
  \label{<+label+>}
\end{equation}
We fix the normalisation such that $\theta^\pm_{f, L, \P^2_L} = L_\infty$ is the line at infinity. This implies the following on intersection numbers 
\begin{equation}
    (\theta^+_{f,L})^2 = (\theta^-_{f,L})^2 = 0 \text{ and } \theta^+_{f,L} \cdot \theta^-_{f,L} = 1.
\end{equation}
We emphasize that this
construction works over any field. Now if $L$ is a complete metrised field, we explain that this is compatible with the
theory of Yuan and Zhang and we can construct canonical Green functions associated to these $b$-divisors to define 
invariant adelic divisors.
\begin{rmq}\label{rmq:}
  Here there is a slight abuse because the topology used in the references mentioned above on Cartier $b$-divisors is
  different from that the one we use here. Call the topology in those references the \emph{weak topology}, the first author
  showed in \cite[Proposition 9.6]{abboudRigidityPeriodicPoints2024} that the boundary topology is stronger than the
  weak topology and that the limiting process used to define $\theta^\pm_{f,L}$ also works with the boundary topology. Namely if $L_\infty$ is the line at infinity in $\P^2_L$ then 
  \begin{equation}
      \frac{1}{\lambda(f)^n} (f^{\pm n})^* L_\infty \xrightarrow[n \rightarrow + \infty]{} \theta_{f,L}^\pm
  \end{equation}
  both with respect to the weak topology and for the boundary topology.
\end{rmq}
\subsection{Green functions and equilibrium measure over a complete metrised field}\label{subsec:}
Let $K_v$ be a complete metrised field.
The construction of the Green functions of a regular plane polynomial automorphism $f$ over the field of complex numbers can be found in \cite{bedfordPolynomialDiffeomorphismsC21991a,HUBBARD1986101}, and over any complete algebraically closed field with a non-trivial
absolute value in \cite{zbMATH06226668}. In \cite{abboudRigidityPeriodicPoints2024}, the first author showed that these constructions are
compatible with the theory of adelic divisors over quasiprojective varieties of Yuan and Zhang.

There is a unique (up to multiplication by a positive constant) Green function $G^\pm_f$ of $\theta_{f,K_v}^{\pm}$ such that
$G^{\pm}_f \circ f^{\pm 1} = \lambda(f)
G^{\pm}_f$. It is defined as follows 
\begin{equation}
  \forall q \in \A^{2, \an}_{K_v}, \quad G^{\pm}_f (q) = \lim_n \frac{1}{\lambda(f)^n} \log^+ \parallel f^{\pm n} (q) \parallel
  \label{<+label+>}
\end{equation}
where $\parallel (x,y) \parallel = \max \left( \left| x \right|, \left| y \right| \right)$. We have the
following properties which can be recovered from \cite{bedfordPolynomialDiffeomorphismsC21991a,zbMATH06226668,zbMATH06288416}. 
\begin{enumerate}
  \item $G_f^{\pm}$ is continuous, $\geq 0$, plurisubharmonic and harmonic over $\left\{ G^{\pm}_f > 0 \right\}$.
  \item There exists an open neighbourhood $W^-$ of $p_-$ in $\P^{2, \an}_{K_v}$ invariant by $f^{-1}$ and such that the
    sequence of functions 
    \begin{equation}
      p \mapsto \frac{1}{\lambda(f)^n} \log^+ \left| \left| f^n (p) \right| \right| - \log^+ \left| \left| p
      \right| \right|
      \label{<+label+>}
    \end{equation}
    is a sequence of continuous functions over $\P^{2, \an}_{K_v} \setminus W^-$ that converges uniformly towards $G^+_f - \log^+
    \left| \left| \cdot \right| \right|$ over $\P^{2, \an}_{K_v} \setminus W^-$.
  \item $G_f^{\pm} \circ f^{\pm 1} = \lambda(f) G^{\pm}_f$.
\end{enumerate}
We write $\overline \theta^\pm_{f, K_v}$ for the induced metrised divisors. We also define $\overline \theta_{f,K_v} =
\overline \theta^+_{f, K_v} + \overline \theta^-_{f, K_v}$ and set $G_f =G^+_f + G^-_f$.
\begin{cor}\label{cor:positive-degree-curve}
  Let $C \subset \A^2_{K_v}$ be a curve that intersects the line at infinity only at $p_+$, then ${G^+_f}_{|C}$ is a
  Green function of ${L_\infty}_{|C}$ where $L_\infty$ is the line at infinity in $\P^2$. In particular, $\overline
  \theta^+_{f, K_v}$ restricts to $\overline C$ to a metrised divisor $\overline D$ where $D = {L_{\infty}}_{|\overline
  C}$.
\end{cor}
From this corollary and following \cite{dujardinDynamicalManinMumford2017} we deduce the following result. 
\begin{prop}\label{prop:restriction-to-curve-has-positive-degree}
  Let $L$ be a finitely generated field over $\Q$ and $f \in \Aut(\A^2_L)$ a regular automorphism defined over $L$.
  Let $C \subset \A^2_L$ be a curve, then $\theta^\pm_{f,L}$ restricts to a $b$-divisor ${\theta^\pm_{f,L}}_{|C}$ and we
  have $\deg ({\theta_{f,L}^\pm}_{|C}) >0$.
\end{prop}
\begin{proof}
  Since $L$ is finitely generated over $\Q$, we can embed $L$ into the field $\C$ of complex numbers. Since the degree
  is defined as the geometric intersection number it is invariant under base change so we can compute over $\C$.
  Consider the canonical Green functions $f$ over $\C$ associated to $\theta_{f,\C}^\pm$ and consider the extension of $f$ to $\P^2$.
  By \cite[Proposition 4.2]{bedfordPolynomialDiffeomorphismsC21991}, there exists an integer $k \geq 0$ such that
  $f^k(C)$ intersects the line at infinity only at $p^+$. By Corollary \ref{cor:positive-degree-curve} we have that
  \begin{equation}
    \deg \left( {\theta_{f,\C}^+}_{|f^k (C)} \right) = L_\infty \cdot \overline{f^k (C)} > 0
  \end{equation}
  where $\overline{f^k (C)}$ is the closure of $f^k (C)$ in $\P^2$.
Now we know by Theorem
\ref{thm:total-mass-measure} that this degree is equal to the total mass of $(dd^c G^+_f)_{|f^k(C)}$. Since $f^k$ induces
an isomorphism between $C$ and $f^k (C)$ we deduce also that $(dd^c G^+_f)_{|C}$ has positive mass which is the degree $\deg
(D^+_{|C})$.
\end{proof}
The \emph{equilibrium measure} of $f$ is defined as 
\begin{equation}
  \mu_f = \frac{1}{2}(dd^c G_f)^2 =  dd^c G^+_f \wedge dd^c G^-_f.
  \label{<+label+>}
\end{equation}
It is clear that $\Supp \mu_f \subset \left\{ G_f = 0 \right\}$ and we have the following characterisation which was
shown first by Dujardin and Favre in \cite{dujardinDynamicalManinMumford2017} and generalised to any
affine surface by the first author in \cite{abboudRigidityPeriodicPoints2024}. We emphasise that this statement holds
over any complete metrised field $K_v$.
\begin{prop}\label{prop:polynomial-convex-hull}
  The compact set $\left\{ G_f = 0 \right\}$ is the polynomial convex hull of $\Supp \mu_f$.
\end{prop}

We end this section by stating three rigidity results. Let $f$ and $g$ be regular  polynomial automorphisms of $\C^2$.

\begin{thm}[\cite{dujardinDynamicalManinMumford2017}]\label{thm:jacobian-1}
  Suppose there exists a curve $C \subset \C^2$ such that $(G^+ - \alpha
  G^-)_{|C}$ is harmonic for some $\alpha > 0$, then $\left| \Jac (f) \right| = 1$.
\end{thm}

\begin{thm}[\cite{dujardinDynamicalManinMumford2017}]\label{thm:dissipativecommoniterate}
    Suppose $f$ is dissipative. If $f$ and $g$ have infinitely many common periodic points, then they share a common iterate.
\end{thm}

The following result comes from Stéphane Lamy's PhD thesis. The argument can be found in
\cite{dujardinDynamicalManinMumford2017} following \cite[Theorem 5.3]{zbMATH01643663}.

\begin{thm}\label{prop:same-measure-common-iterate}
  If $f, g \in \Aut (\C^2)$ are regular automorphisms such that $\mu_f = \mu_g$, then there exists $N,M \neq 0$ such
  that $f^N = g^M$.
\end{thm}

\subsection{Over a number field}\label{subsec:}
If $f$ is defined over a number field $K$, we can construct the Green functions for any place $v \in \cM(K)$. This
defines two $f$-invariant strongly nef adelic divisors $\overline \theta_{f,K}^+, \overline \theta_{f,K}^- \in \hat \Div(\A^2_K /
\OO_K)$ such that 
\begin{equation}
  f^* \overline \theta_{f,K}^\pm = \lambda(f)^{\pm 1} \overline \theta_{f,K}^\pm.
  \label{}
\end{equation}
They are unique up to multiplication by a positive constant. We define $\overline \theta_{f,K} = \overline \theta_{f,K}^+ +
\overline \theta_{f,K}^-$.
The \emph{canonical height} of $f$ is $\hat{h}_f \coloneqq h_{\overline \theta_{f,K}}$ and we have that 
\begin{equation}
  \mu_{f,v} = \frac{1}{2} c_1 (\overline \theta_f)_v^2 = \frac{1}{2} dd^c G^+_{f,v} \wedge dd^c G^-_{f,v}.
  \label{<+label+>}
\end{equation}
In particular, the canonical height $h_f$ satisfies the \emph{Northcott property} which implies that $q \in \overline
K^2$ is $f$-periodic if and only if $h_{f,K} (q) = 0$ (\cite{zbMATH06226668}).

More precisely, consider the divisor $D = L_\infty$ which is the line at infinity in $\P^2_K$. Consider the projective model 
\begin{equation}
    \sX = \P^2_{\OO_K}
\end{equation}
over $\OO_K$ and define $\overline \sD$ as the model adelic divisor where $\sD$ is the closure of $L_\infty$ in $\P^2_{\OO_K}$ and the Green function $G_{\overline \sD}$ over $\P^2 (\C) \setminus L_\infty = \C^2$ is 
\begin{equation}
    G_{\overline \sD} (x,y) = \log^+ \max ( |x|, |y|).
\end{equation}
Then, from \cite{abboudRigidityPeriodicPoints2024}, we have with respect to the boundary topology of Yuan and Zhang that 
\begin{equation}
    \frac{1}{\lambda(f)^n} (f^{\pm n})^* \overline \sD \xrightarrow[n \rightarrow + \infty]{} \overline \theta_{f,K}^\pm.
\end{equation}

\subsection{Over a function field}\label{subsec:over-function-field}
If $f$ is defined over the function field $K = K(B)$ of a projective curve $B$ over some field $k$, then as in the number field case, the construction of the Green functions of $f$ works for any $b \in \cM(K)$. This defines also strongly nef adelic divisors
$\widetilde \theta_{f,K}^\pm \in \hat \Div(\A^2_{K(B)})$ and the \emph{geometric canonical height} of $f$ is defined as 
\begin{equation}
  h_{f,K}^{geom} = h_{\widetilde \theta^+_{f,K} + \widetilde \theta^-_{f,K}}.
  \label{<+label+>}
\end{equation}
More precisely, if $D = L_\infty$ is the line at infinity in $\P^2_K$ then define $\sX = \P^2_k \times B$ and let $\widetilde \sD$ be the pullback of the line at infinity in $\P^2_k$ then $(\sX, \widetilde \sD)$ is a projective model of $(\P^2_K, L_\infty)$ and from \cite{abboudRigidityPeriodicPoints2024} we have with respect to the boundary topology of Yuan and Zhang that 
\begin{equation}
    \frac{1}{\lambda(f)^n} (f^{\pm n})^* \widetilde \sD \xrightarrow[n \rightarrow +\infty]{} \widetilde \theta_{f,K}^\pm.
\end{equation}
\subsection{In a family}\label{subsec:in-a-family}
Suppose now that $f : \A^2 \rightarrow \A^2$ is a regular automorphism defined over the function field of a
projective curve $B$ defined over a number field $K$. Then, up to taking an open subset $B' \subset B$ we have that
$f$ extends to an automorphism of $\A^2_K \times B'$. The construction of the Green functions also works in this setting
and was also done in \cite{abboudRigidityPeriodicPoints2024}. We also have the existence of two strongly nef effective
adelic divisors $\overline \theta^\pm_f \in \hat \Div (\A^2_K \times B' / \OO_K)$, their Green functions are the
\emph{fibered Green functions of $f$}. For every $b \in B' (\overline K)$, the adelic divisors $\theta_f^\pm$ restrict on
the fiber $\A^2_K \times \left\{ b \right\}$ to the adelic divisors $\overline \theta^\pm_{f,K(b)} \in \hat \Div (\A^2_K
\times \left\{ b \right\} / \OO_{K (b)})$ constructed in the previous paragraph. In particular, the Green functions are
equal to 
\begin{equation}
  \forall (x,y,b) \in \A^2 \times B (\overline K), \quad G^\pm_{f} ( (x,y, b)) = G_{f_b}^\pm (x,y).
  \label{<+label+>}
\end{equation}
In particular, 
\begin{equation}
  \forall (x,y,b) \in A^2 \times B (\overline K), \quad h_f ( (x,y, b)) = h_{f_b} ((x,y))
  \label{<+label+>}
\end{equation}
and $((x,y),b)$ is $f$-periodic if and only if $(x,y)$ if $f_b$-periodic.
More precisely, $X = \P^2_K$ and $D=L_\infty$ be the line at infinity. Let $\sB$ be a projective model of $B$ over $\OO_K$ and $\sX = \P^2_{\OO_K} \times \sB$. Define $\overline \sD$ as follows: $\sD$ is the closure of $L_\infty$ in $\sX$ and the Green function $G_{\overline \sD}$ over $\P^2(\C) \times B(\C) \setminus (L_\infty \times B(\C) )  = \C^2 \times B(\C)$ is equal to 
\begin{equation}
    G_{\overline \sD}(x,y,b) = \log^+ \max (|x|, |y|).
\end{equation}
From \cite{abboudRigidityPeriodicPoints2024}, we have with respect to the boundary topology of Yuan and Zhang that 
\begin{equation}
    \frac{1}{\lambda(f)^n} (f^{\pm n})^* \overline \sD \xrightarrow[n \rightarrow + \infty]{} \overline \theta_f^\pm.
\end{equation}
Now, we also have that $\widetilde \sX = \P^2_K \times B$ is a projective model of $\P^2_{K(B)}$ over $B$ and the divisor $\sD$ restricts to the divisor $\widetilde \sD$ from \S \ref{subsec:over-function-field} above. So we also recover the geometric canonical height of $f$ from this construction and in particular, 
\begin{equation}
    \deg_{\overline \theta^+_f + \overline \theta^-_f} (\A^2_K \times B') = (\widetilde \theta^+_{f, K(B)} + \widetilde \theta^-_{f,K(B)})^3.
\end{equation}

Finally, if we restrict these limits of divisors over $\A^2_{K(B)}$, i.e on the generic fiber of $\A^2_K \times B$, then we recover the $b$-divisors $\theta^\pm_{f, K(B)}$ associated to $f : \A^2_{K(B)} \rightarrow \A^2_{K(B)}$ discussed in the beginning of this section.

\section{Geometric height and an inequality}\label{sec:geometric-height}
We state and prove the next proposition for geometric adelic divisors and line bundles but it holds verbatim for
arithmetic adelic divisors and line bundles with the exact same proof.

\begin{prop}\label{prop:inequality-intersection-number}
  Let $K$ be the function field of a projective curve over some field. Let $U$ be a quasiprojective variety of dimension
  $n$ over $K$ and let $\widetilde L_1, \dots, \widetilde L_n$ be strongly nef geometric adelic line bundles and
  $\widetilde D$ an effective geometric adelic divisor, then 
  \begin{equation}
    \widetilde D \cdot \widetilde L_1 \cdots \widetilde L_n \geq \sum_{v \in \cM(K)} n_v \int_{U^{\an,v}} G_{\widetilde D, v} c_1 (\widetilde L_1)_v
    \cdots c_1 (\widetilde L_n)_v.
    \label{<+label+>}
  \end{equation}
\end{prop}
\begin{proof}
  Let $\widetilde D_i$ and $\widetilde L_{k,i}$ be sequence of model divisors and line bundles converging respectively to
  $\widetilde{D}$ and $\widetilde{L}_k$. We can choose in particular $D_i$ effective and $\widetilde L_{k,i}$ nef.
  Define for every $v, \mu_{i,v} = c_1 (\widetilde L_{1,i})_v \cdots c_1 (\widetilde L_{n,i})_v$ and $\mu_v = c_1
  (\widetilde L_1)_v \cdots c_1 (\widetilde L_n)_v$.
  We have that 
  \begin{equation}
    \widetilde D \cdot \widetilde L_1 \cdots \widetilde L_d = \lim_i \widetilde D_i \cdot \widetilde L_{1,i} \cdots \widetilde
    L_{d,i}.
    \label{<+label+>}
  \end{equation}
  By \eqref{eq:intersection-number}, we have that 
  \begin{equation}
    \widetilde D_i \cdot \widetilde L_{1,i} \cdots \widetilde L_{n,i} \geq \sum_v \int_{U^{\an,v}} G_{\widetilde D_i}
    \mu_{i,v}.
    \label{ }
  \end{equation}
  Now, for every $v \in \cM(K)$, let $\phi_v$ be a continuous function over $U^{\an,v}$ with compact support $\Omega_v$
  such that $0 \leq \phi \leq 1$. Then we have 
  \begin{equation}
    \widetilde D_i \cdot \widetilde L_{1,i} \cdots \widetilde L_{n,i} \geq \sum_v \int_{U^{\an,v}} \phi_v G_{\widetilde D_i} \mu_{i,v}.
    \label{<+label+>}
  \end{equation}
  Now, $\phi_v G_{\widetilde D_i}$ has support over the compact subset $\Omega_v$ and since $\mu_{i,v}$ weakly converges
  towards $\mu_v$ we have that 
  \begin{equation}
    \widetilde D \cdot \widetilde L_1 \cdots \widetilde L_n \geq \sum_v \int_{U^{\an,v}} \phi_v G_{\widetilde D, v} \mu_v.
    \label{<+label+>}
  \end{equation}
  We apply this inequality with a compact exhaustion $\phi_{n,v}$ of $U^{\an,v}$ to get the result by the monotone
  convergence theorem. The existence of a compact exhaustion is classical if $v$ is archimedean and if $v$ is
  non-archimedean this exists by \cite{chambert-loirFormesDifferentiellesReelles2012}.
\end{proof}

Let $f: B\times \A^2\to B\times \A^2$ be a family of regular plane polynomial automorphisms over some field $K$ of characteristic zero.
The set of periodic points $\Per(f_\eta)$ and the set of zero height $\{\Tilde{h}_f=0\}$ coincide except for the trivial case. We say that a family of regular plane polynomial automorphisms $f$ is \emph{isotrivial} if there exists a finite map $S_f: B' \to B$ of curves and a family of affine maps $(\varphi_{f,b})_{s\in B'}$ such that $\varphi^{-1}_{f,b} \circ f_{S_f(b)}\circ \varphi_{f,b} $ is defined over $\overline{K}$. Given two families $f$ and $g$, they are \emph{simultaneously isotrivial} if we can choose $S_f=S_g$ and $\varphi_f=\varphi_g$.
\begin{thm}[\cite{zbMATH07714404}]\label{thm:GVhenon}
    If $f$ is a non-isotrivial family of regular plane polynomial automorphisms, then $\Per(f_\eta)=\{\Tilde{h}_f=0\}$.
\end{thm}

\begin{prop}\label{prop:current-not-vanishing}
  Let $f$ and $g$ be two families of dynamically independent regular plane polynomial automorphisms defined over the function field of a
  projective curve $B$ over a number field $K$. Using the notations of \S \ref{subsec:in-a-family}, define $\overline R = \overline \theta_f + \overline \theta_g$, then $R^3 = 0$ if and only if the families $f$ and $g$ are simultaneously isotrivial.
\end{prop}
\begin{proof}
  Since $\overline R = \overline \theta_f + \overline \theta_g$ we have that $\overline R$ is strongly nef so that
  $R^3 \geq 0$.
  The condition being clearly sufficient, we prove that it is also necessary. 
Now recall that by \S \ref{subsec:in-a-family} we have that 
  \begin{equation}
      R^3 = (\widetilde \theta_f + \widetilde \theta_g)^3.
  \end{equation}  
    We may suppose that the two families $f$ and $g$ are defined over $B_\C = B \times_{\spec K} \spec \C$ by a base
  change. By Proposition
  \ref{prop:inequality-intersection-number}, used in the geometric setting we have
  \begin{equation}
 R^3 \geq \sum_{b \in \cM(\C(B))} \int_{X^{\an,b}} G_{f,b} (dd^c G_{g,b})^2 + \int_{X^{\an,v}}
    G_{g,b} (dd^c G_{f,b})^2.
    \label{}
  \end{equation}
  If the intersection product is zero, then we get that for every $b \in \cM(\C(B)), \supp \mu_{f,b} \subset \left\{ G_{g,b}
  = 0 \right\}$. By Proposition \ref{prop:polynomial-convex-hull} this implies that $\left\{ G_{f,b} = 0 \right\}
  \subset \left\{ G_{g,b} = 0 \right\}$, and by symmetry that
  \begin{equation}
    \forall b \in \cM(\C(B)), \quad \left\{ G_{f,b} = 0 \right\} = \left\{ G_{g,b} = 0 \right\}.
    \label{<+label+>}
  \end{equation}
  By the height decomposition~\eqref{eq:height}, we have
  \begin{equation}
    \left\{ \tilde h_f = 0 \right\} = \left\{ \tilde h_g = 0 \right\}.
    \label{<+label+>}
  \end{equation}

  If neither $f$ nor $g$ is isotrivial, then $\mathrm{Per}(f_\eta) = \mathrm{Per}(g_\eta)$ by Theorem~\ref{thm:GVhenon}.
  If one of them, say $f$, is isotrivial while $g$ is not, then the set $\left\{\Tilde{h}_f = 0 \right\}$ is
  uncountable, whereas $\left\{\Tilde{h}_g = 0 \right\}$ is countable. Now suppose that both $f$ and $g$ are isotrivial.
  In this case, we may assume that $f$ is defined over $\C$.
  Suppose, for contradiction, that $g$ is not trivial.
  There exists a constant point $x \in \A^2(\overline{\C(B)})$ that is not $g_\eta$-periodic. Otherwise the set of
  $\overline {\C(B)}$-points defined over $\C$ would be $g_\eta$ invariant and that would imply that $g$ is defined over
  $\C$ and thus trivial.
  Then $\hat{h}_g(x) > 0$, whereas $\hat{h}_f(x) = 0$. Thus none of three possibilities above is possible and f and g have to be simultaneously isotrivial.
\end{proof}

Consider the fiber product $B\times\A^2\times\A^2$. Denote by $p_1$ and $p_2$ the projections to the first and the second factors $B\times \A^2$. Define
\begin{align}\label{eq:Zper}
    Z_\varepsilon \coloneqq \{(b,x,y)\in B\times\A^2\times\A^2 \mid \{\hat{h}_{f_b}(x)+\hat{h}_{g_b}(y) \leq \varepsilon \}.
\end{align}

\begin{prop}\label{prop:non-zariski-dense-Zper}
    Let $f$ and $g$ be two families of dynamically independent regular plane polynomial automorphisms defined over a curve $B$ over a number field
    $K$. Then there exists a small constant $\varepsilon>0$ such that  $Z_\varepsilon(\overline{K})$ is not Zariski dense.
\end{prop}
\begin{proof}
    We may suppose that $f$ and $g$ are not simultaneously isotrivial and invoke thus
    Proposition~\ref{prop:current-not-vanishing} and Theorem \ref{thm:total-mass-measure}. The adelic divisor $\overline R = \overline \theta_f + \overline \theta_g$ satisfies $R^3 > 0$. Define the adelic divisors 
    \begin{equation}
        \overline R_h = p_1^* \overline R + p_2^* \overline \theta_h
    \end{equation}
    for $h = f,g$.
    By Proposition \ref{prop:intersection-number-fibration}, we have that 
    \begin{equation}
        R_h^5 \geq (p_1^* R)^3 \cdot (p_2^* \theta_h)^2 = R^3 \cdot (\theta_{h, K(B)})^2 = R^3 > 0.
    \end{equation}
    We can therefore apply Theorem~\ref{thm:arithmetic-equidistribution} to $\overline R_h$. Fix an archimedean place $v$, we apply the equidstribution result only over this place so we drop the index $v$ in the notations. Suppose that the proposition does not hold, then there exists a sequence of closed points $p_n := (b_n, x_n,y_n) \in B \times \A^2 \times \A^2 (\overline K)$ such that $h_{\overline R_f}(p_n) \rightarrow 0$ and $h_{\overline R_g} (p_n) \rightarrow 0$. The two measures $c_1(\overline R_f)^5$ and $c_1 (\overline
    R_g)^5$ over $B(\C) \times \C^2 \times \C^2$ are then proportional and they are proportional to
    \begin{equation}
        \text{ for } h = f,g, \quad p_1^* c_1 (\overline R)^3 \wedge p_2^* c_1 (\overline \theta_h)^2.
    \end{equation}
    Therefore, for
    $\pi_*c_1(\overline{R})^3$-almost every parameter $b\in B(\C)$, we have $\mu_{f_b}=\mu_{g_b}$, and $f_b$ and $g_b$
    share a common iterate by Proposition~\ref{prop:same-measure-common-iterate}. However, the set of parameters $b \in
    B(\C)$ where $f_b$ and $g_b$ share a common iterate is countable and is not charged by the measure
    $\pi_*c_1(\overline{R})^3$, hence a contradiction. Therefore the set $Z_\varepsilon(\overline{K})$ is not Zariski
    dense for some $\epsilon > 0$.
\end{proof}

\section{Proof of Theorem~\ref{bigthmA} and~\ref{bigthmB} }
This section is devoted to the proof our main Theorems. We actually prove a stronger version.
\begin{thm}\label{thm:small-height}
    Let $f$ and $g$ be two dynamically independent families of regular plane polynomial automorphisms. Suppose that the set 
    $\{|\Jac(f)|=1 \} \cap \{|\Jac(g)|=1 \}$ is discrete (e.g., $f$ is a dissipative family) with respect to the analytic topology of $B(\C)$. Then there exist $\varepsilon>0$, a Zariski open subset $U$ of $B$ and $D > 0$ depending only on $f$ and $g$ such that for every $b \in U(\overline \K)$,  we have
  \begin{equation}
    \# \{z\in \A^2(\overline K) \mid \hat{h}_{f_b}(z)+\hat{h}_{g_b}(z) \leq \varepsilon \} \leq D.
    \label{eq:uniformhight}
  \end{equation}
\end{thm}
\begin{proof}
By a slight abuse of notation, all projections to the base $B$ will be denoted by $\pi$. We would like to find a Zariski open subset $U\subset B$ so that Eq~\eqref{eq:uniformhight} holds for any $b\in U(\overline{K})$.

Choose $\varepsilon$ satisfying the conclusion of Proposition~\ref{prop:non-zariski-dense-Zper} and
denote by $Z$ the Zariski closure of $Z_{\varepsilon}$. If the projection of $Z$ to $B$ is not dominant, then it suffices to set $U\coloneqq B\setminus\pi(Z)$ and take $D=0$. 

We may therefore assume that $Z$ is irreducible and that $\pi(Z)=B$. By Proposition~\ref{prop:non-zariski-dense-Zper}, $Z$ is not Zariski dense. Denote by $p_Z$ the restriction of $p_1$ to $Z$. 

\textbf{Case (i): $p_Z$ is dominant}.\ The generic flatness and Proposition~\ref{prop:independent-intersection-number-fiber} yield a family of (possibly reducible) subvarieties $Z_{(b,x)}\coloneqq p_Z^{-1}(b,x)$ parameterised by a Zariski open subset $V$ of $B\times\A^2$, all of fixed degree $D'\geq 1$. If $\dim Z=3$, then $Z_{(b,x)}$ are points and we are done. If $\dim Z=4$, then $Z_{(b,x)}$ are curves. Write $W = \A^2 \times B \setminus V$ and write $W = W^h \cup W^v$ its decomposition into horizontal and vertical parts with respect to $\pi$. Set $U = B \setminus \pi(W^v)$, each irreducible component of $W^h$ yields a family of points or curves of uniform bounded degree over $U$. Now take any parameter $b\in U(\overline{K})$. There are three possibilities. 
\begin{enumerate}
    \item There exists no \emph{$\epsilon$-small height}, i.e., $\hat{h}_{f_b}(x)+ \hat{h}_{g_b}(x) \leq \varepsilon$;
    \item There exists $x \in \A^2_b $ of $\epsilon$-small height such that $(b,x) \in V$. By the construction, all $\varepsilon$-small height points $y\in \A^2_b$ are contained in $Z_{(b,x)}$;
    \item All $\epsilon$-small height points are contained in $W_h \cap \A^2_b$.
\end{enumerate} 

\textbf{Case (ii): $p_Z$ is not dominant}. Then we consider its image $Z'\coloneqq p_Z(Z)\subset B\times \A^2$. We have a family $(Z'_b)_b$ of  subvarieties parameterised by $B$. Either it is a family of points (\textbf{Case (ii.a)}) then Theorem~\ref{thm:small-height} is finished, or it is a family of curves with fixed degree (\textbf{Case (ii.b)}). 

In general we just decompose $Z$ as a union of irreducible components and do the above analysis component by component. Since no condition on the Jacobians has been used so far, we have thus proved the following.
\begin{thm}\label{mainthm2}
    Let $f$ and $g$ be two dynamically independent families of regular plane polynomial automorphisms. Then there exist positive constants $\varepsilon>0$ and $D' > 0$, depending only on $f$ and $g$, such that for all but finitely many $b \in B(\overline \K)$, the set of small heights
  \begin{equation}
    \{z\in \A^2(\overline K) \mid \hat{h}_{f_b}(z)+\hat{h}_{g_b}(z) \leq \varepsilon \}
  \end{equation}
  is contained in a (possibly reducible) curve of degree $D'$.
\end{thm}

We continue our analysis and first go back to the case (i) $\dim Z=4$. Under the additional assumption on the Jacobians, we show that this case is in fact not possible. Let $\overline{L}_g^\pm$ be the strongly nef adelic divisor defined by $\overline{L}_g^\pm\coloneqq
p_1^*\overline{R}+p_2^*\overline{\theta}_g^\pm$. There
exists a generic sequence $z_n\coloneqq(b_n,x_n,y_n)\in Z(\overline{K})$ such that $\hat{h}_{\overline{L}_g^\pm}(z_n)\to
0$. To apply Theorem~\ref{thm:arithmetic-equidistribution}, we need to show that the restriction of $\overline L_g^\pm$
to $Z$ is big. This follows from proposition \ref{prop:restriction-to-curve-has-positive-degree}. Indeed, we need to
show that the geometric intersection number
\begin{equation}
    (L_g^\pm)^4_{|Z} \geq ( p_1^* R^3 \cdot p_2^* \theta_g^\pm)_{|Z} >0.
\end{equation}
The restriction of $Z$ to the generic fiber of $p_1$ is a curve $C$ and $p_2^* \theta_g^\pm$ restricts to the
$b$-divisor $\theta_{g, L}^\pm$ associated to the regular automorphisms induced by $g$ on the generic fiber $\A^2_L$
where $L$ is the function field of $\A^2_K \times B$. By Corollary
\ref{cor:intersection-number-fibration-quasiprojective} we get that 
\begin{equation}
  ( p_1^* R^3 \cdot p_2^* \theta_g^\pm)_{|Z} = R^3 \cdot \deg({\theta_{g, L}^{\pm}}_{|C})
\end{equation}
which is a positive number by Corollary \ref{cor:positive-degree-curve}.
By Theorem \ref{thm:arithmetic-equidistribution}, the two measures $c_1(\overline{L}_g^\pm)^5$ must then be proportional
on $B\times\A^2\times\A^2(\C)$. It follows that, for $\pi_*(c_1(\overline{R}))^3$-almost every parameter $b\in B(\C)$,
the measures $\ddc G_{g_b}^\pm|_{Z_b}$ are proportional, and moreover that $|\Jac(g_b)|=1$ by
Theorem~\ref{thm:jacobian-1}. By symmetry, the same conclusion holds for $f_b$, so that $|\Jac(f_b)|=1$ almost surely
for the same measure $\pi_*(c_1(\overline{R}))^3$. However, this measure does not charge discrete sets, which yields a
contradiction. 

Now let us look at Case (ii.b). We have a family of curves $Z'$ over $B$ that contains all the common periodic points. We need the following lemmas.

\begin{lemme}[Siu's inequality for $b$-divisor]
    Let $U$ be a $d$-dimensional quasi-projective variety over a field $k$ of characteristic 0. Let $\alpha$ be a nef $b$-divisors and $\beta$ a big and nef $b$-divisor. We have
    \begin{align}\label{eq:siu}
        \alpha \leq d \frac{\alpha\cdot \beta^{d-1}}{\beta^d} \beta.
    \end{align}
\end{lemme}
\begin{proof}
    Over a projective variety, this is the classical Siu's inequality, see for example \cite{dangDegreesIteratesRational2020}. It follows for $b$-divisors by a limit argument.
\end{proof}

\begin{lemme}\label{lem: Z'non-degenerate}
    $Z'$ is non-degenerate: $R^2 \cdot Z' >0$.
\end{lemme}
\begin{proof}[Proof of Lemma \ref{lem: Z'non-degenerate}]
    $R$ is big by Proposition~\ref{prop:current-not-vanishing}. Take any ample divisor $H$ in $B$. Siu's inequality~\eqref{eq:siu} implies that $3\frac{\pi^*H \cdot R^2}{R^3} R \geq \pi^*H$ and we have that $\pi^* H \cdot R^2 > 0$ by Corollary \ref{cor:intersection-number-fibration-quasiprojective}. We infer that
    \begin{align*}
        R^2\cdot Z' \geq R\cdot \pi^*H \cdot Z' = \deg H  R_\eta \cdot Z'_\eta >0,
    \end{align*}
    where the last inequality comes from Proposition~\ref{prop:restriction-to-curve-has-positive-degree}.
\end{proof}
Define strongly nef adelic divisors $\overline{M}^\pm_g\coloneqq p_1^*\overline{R} + p_2^* \overline{\theta}^+_g$. Applying Lemma~\ref{lem: Z'non-degenerate}, a similar computation as in Case (i) implies that $M_g^\pm$ is big on $Z$: $(M_h^+)^3 \cdot Z >0$. Thus the two measures $c_1(\overline{M}^\pm_g)^3 \wedge [Z]$ are proportional and for $\pi_*\left(c_1(\overline{R})^2\wedge [Z']\right)$-almost every parameter $b$, the measures $\ddc G_{g_b}^\pm|_{Z_b}$ are proportional. We infer again that $|\Jac(g_b)|=1$ and we conclude as in Case (i).

The only possible case is Case (ii.a) and the proof of Theorem~\ref{thm:small-height} is complete.
\end{proof}

\begin{proof}[End of proof of Theorem \ref{bigthmB}].
To finish the proof of Theorem~\ref{bigthmB}, we need to show that, up to taking a larger $D$, for any transcendental $b\in B(\C \setminus \overline{K})$, Equation~\eqref{eq:uniformperiodicpoints}
holds. Denote by $D_\eta$ the number of common periodic points of $f_\eta$ and $g_\eta$ at the generic fiber. Denote by $\kappa_b$ the residue field of $b$. Then there is a field embedding of $\overline{K}(B) \to \kappa_b$. Hence the number of common periodic points of $f_b$ and $g_b$ is also equal to $D_\eta$. Finally we take the new constant to be $\max\{D_\eta, D\}$. This ends the proof of Theorem~\ref{bigthmB}.
\end{proof}

\begin{proof}[End of proof of Theorem \ref{bigthmA}]
To complete the proof of Theorem~\ref{bigthmA}, it remains to consider the parameters $b$ for which there are infinitely many common periodic points. By Theorem~\ref{thm:dissipativecommoniterate}, this implies that $f_b$ and $g_b$ share a common iterate: there exist integers $N_b$ and $M_b$ such that $f_b^{N_b}=g_b^{M_b}$. Since the set of such parameters is finite, we may choose $N_b$ and $M_b$ independently of $b$. This concludes the proof of Theorem~\ref{bigthmA} in the dynamically independent case. If, on the other hand, $f$ and $g$ are dynamically dependent, then by definition they globally share a common iterate.
\end{proof}

\bibliographystyle{alpha}
\bibliography{biblio}
\end{document}